\documentclass[12pt, reqno]{amsart}
\usepackage{amsmath, amsthm, amscd, amsfonts, amsthm, amssymb, graphicx}
\usepackage[bookmarksnumbered, colorlinks, plainpages]{hyperref}

\newtheorem{theorem}{Theorem}[section]
\newtheorem{lemma}[theorem]{Lemma}

\newtheorem{corollary}[theorem]{Corollary}
\theoremstyle{definition}

\newtheorem{example}[theorem]{Example}

\newtheorem{remark}[theorem]{Remark}
\numberwithin{equation}{section}

\begin{document}
	
	\title[Some extensions from famous theorems]{Some extensions from famous theorems for $h$-mid-convex function}
	
	\author[Garejelo, Mirzapour, Morassaei]{Amir Garejelo$^1$, Farzollah Mirzapour$^1$ and Ali Morassaei$^1$}
	\vspace{-2cm}
	\address{$^{1}$ Department of Mathematics, Faculty of Sciences, University of Zanjan, P. O. Box 45195-313, Zanjan, Iran.}
	
	\email{amirgharajelo@znu.ac.ir}
	\email{f.mirza@znu.ac.ir}
	\email{morassaei@znu.ac.ir}
	
	
	
	
	\begin{abstract}
	In this paper, we prove that every continuous $h$-mid-convex with suitable conditions on $h$ is $h$-convex function. Also, we extend Ostrowski theorem, Blumberg-Sierpinski theorem, Bernstein-Doetsch theorem, Mehdi theorem.
\end{abstract} \maketitle
\section{Introduction}
Throughout this paper, we assume that $I, J\subseteq\mathbb{R}$ are intervals such that $0\in J$. A function $f:I\longrightarrow\mathbb{R}$ is called \textit{convex function}, if for every $x, y\in I$ and $\lambda\in [0, 1]$ the following inequality holds:
$$
f\big(\lambda x+(1-\lambda)y\big)\le\lambda f(x)+(1-\lambda)f(y)\,.
$$
The function $f$ is said to be \textit{concave}, if $-f$ is convex. The function $f$ is called \textit{J-convex (Jensen convex function)} or \textit{mid-convex}, if for every $x,y \in I$,
$$
f\left(\frac{x+y}{2}\right)\le\frac{f(x)+f(y)}{2}\,.
$$
We say that $\sum_{i=1}^{n}\lambda_ix_i$ is an \textit{rational convex combination} of the points $x_1,\cdots,x_n\in I$, if $\lambda_i$ is non-negative rational number for $i=1,\cdots,n$ and $\sum_{i=1}^{n}\lambda_i=1$.

It is obvious that $f$ is mid-convex on $I$, if and only if for any rational convex combination of points in $I$, 
\begin{equation}\label{eq1.1}
	f\left(\sum_{i=1}^{n}\lambda_ix_i\right)\le\sum_{i=1}^{n}\lambda_if(x_i)\,.
\end{equation}
Inequality \eqref{eq1.1} is called \textit{Jensen's inequality}. We know that if $f$ is mid-convex, then such functions need not be continuous, but that when they are, they will be convex.

Chademan and Mirzapour in year 1999 \cite{CM}, defined the mid-convex function on a topological group and then extended the theorems of Bernstein-Doetsch, and Ostrowski, concerning the continuity of mid-convex functions are extended to open subsets of locally compact and root-approximable topological groups.

Let us recall definitions of some special classes of functions.

We say that \cite{GL} $f:I\to\mathbb R$ is a \textit{Godunova-Levin function} or that $f$ belongs to the class $Q(I)$ if $f$ is non-negative and for all $x, y\in I$ and $t\in(0,1)$ we have
$$
f(t x+(1-t)y) \le \frac{f(x)}{t}+\frac{f(y)}{1-t}\,.
$$
For $s\in(0, 1]$, a function $f:[0, \infty)\to[0, \infty)$ is said to be \textit{$s$-convex}, or that $f$ belongs to the class $K_s^2$, if 
$$
f(t x+(1-t)y) \le t^sf(x)+(1-t)^sf(y)
$$
for every $x, y\in [0, \infty)$ and $t\in[0, 1]$ (see \cite{Bre}). Also, we say that $f:I\to[0, \infty)$ is a \textit{$P$-function} \cite{DPP}, or that $f$ belongs to the class $P(I)$, if for all $x, y\in I$ and $t\in[0, 1]$ it holds that
$$
f(t x+(1-t)y) \le f(x)+f(y)\,.
$$

In \cite{VAR}, Varo\v{s}anec defined  the $h$-convex function as follows:\\
Let $h:J\to\mathbb{R}$ be a non-negative function, $h\not\equiv 0$. We say that $f:I\to\mathbb{R}$ is a \textit{$h$-convex function}, or that $f$ belongs to the class $SX(h, I)$, if $f$ is non-negative and for all $x, y\in I$, $t\in(0, 1)$ we have 
\begin{equation}\label{eq1111}
	f(t x+(1-t)y) \le h(t)f(x)+h(1-t)f(y)\,.
\end{equation}
If inequality \eqref{eq1111} is reversed, then $f$ is said to be \textit{$h$-concave}, that is, $f\in SV(h, I)$. 

If $h(t)=t$, then $SX(h, I)$ is exactly equal to the set of all non-negative convex functions defined on the interval $I$, and the set of all non-negative concave functions is equal to $SV(h, I)$. If $h(t)=\frac{1}{t}$, then $SX(h, I)=Q(I)$; if $h(t)=1$, then $SX(h, I)\supseteq P(I)$; and if $h(t)=t^s$, where $s\in(0, 1)$, then $SX(h, I)\supseteq K_s^2$.

Also, $f$ is called \textit{$h$-mid-convex function}, if for every $x, y\in I$, the following inequality holds:
$$
f\left(\frac{x+y}{2}\right)\le h\Big(\frac{1}{2}\Big)\big(f(x)+f(y)\big)\,.
$$

If the interval  $J$ is  closed under addition, then a function $h:J\to\mathbb{R}$ is said to be a \textit{super-additive function} if 
\begin{equation}\label{SA}
	h(x+y) \ge h(x)+h(y)\,,
\end{equation}
for all $x, y\in J$. If inequality \eqref{SA} is reversed, then $h$ is said to be a \textit{sub-additive function}. If the equality holds in \eqref{SA}, then $h$ is said to be an \textit{additive function}.

If the domain of a function $h$ is closed under multiplication, then  $h$ is called a \textit{super-multiplicative function} if 
\begin{equation}\label{SM}
	h(xy) \ge h(x)h(y)\,,
\end{equation}
for all $x, y\in J$ \cite{VAR}. If inequality \eqref{SM} is reversed, then $h$ is called a \textit{sub-multiplicative function}. If the equality holds in \eqref{SM}, then $h$ is called a \textit{multiplicative function}.
\begin{example}\cite{VAR}
	Consider the function $h:[0, +\infty)\to\mathbb{R}$ by $h(x)=(c+x)^{p-1}$, where $p$ is a non-negative real number. If $c=0$, then the function $h$ is multiplicative. If $c\ge 1$, then for $p\in(0, 1)$ the function $h$ is super-multiplicative and for $p>1$ the function $h$ is sub-multiplicative.
\end{example}
Sarikaya et al. presented some inequalities Hadamard-type for $h$-convex functions. They also proved some Hadamard-type inequalities for products of two $h$-convex functions \cite{SSO}.

Dragomir gave some inequalities of Hermite-Hadamard type for $h$-convex functions defined on convex subsets in real or complex linear spaces. Also, he provides applications for norm inequalities \cite{D1}.

Darvish et al. \cite{DNT} introduced the concept of operator $h$-convex functions for positive linear maps, and proved some Hermite-Hadamard type inequalities for these functions. As applications, they obtained several trace inequalities for operators.

In 2022 \cite{AMM}, Abbasi and et al. give some characterizations of $h$-convex functions and then generalize the Jensen-Mercer inequality for these functions. Also, the authors present a refinement of Jensen's inequality for $h$-convex function.

Morassaei and Mirzapour in \cite{MM} stated some integral inequalities for $h$-log-convex functions involving classical Hermite–Hadamard inequality. Also, by using notations of Mond–P\v{e}cari\'{c} method in operator inequalities, they extended some inequalities for operator $h$-log-convex functions.

In this paper, we show that any $h$-mid-convex function under suitable conditions is $h$-convex. Then, we extend Jensen's theorem, Bernstein-Doetsch theorem, Blumberg-Sierpinski theorem, Ostrowski theorem, and Mehdi theorem for $h$-mid-convex function.

\section{Results on $h$-mid-convexity and $h$-convexity}
We know that every continuous mid-convex function is a convex function \cite{RV}. In this section, we give two proofs for the generalization of this statement for $h$-mid-convex function as follows.
\begin{theorem}
	Let $h: \mathbb{R}^+\to \mathbb{R}^+$ be a super-multiplicative and continuous function such that $h(t)\geq t$ and $ f $ be a continuous $h$-mid-convex function on $[a,b]$, then $f$ is $h$-convex.
\end{theorem}
\begin{proof}
	We prove the theorem in two cases:\\
	\textbf{Case 1.} For $1\leq i\leq n$, set $\lambda_i=\dfrac{1}{n}$ and by induction, we show that 
	\begin{equation}\label{eq11}
		f\left(\dfrac{1}{n}\sum_{i=1}^nx_i\right)\leq h\big(\dfrac{1}{n}\big)\sum_{i=1}^nf(x_i)\,.
	\end{equation}
	For $ n=2 $, we have
	\begin{align*}
		f(\lambda_1x_1+\lambda_2x_2)=f\left(\dfrac{x_1+x_2}{2}\right)\leq h(\dfrac{1}{2})(f(x_1)+f(x_2))\,,
	\end{align*}
	this inequality is obvious.
	
	We assume that the inequality \eqref{eq11} holds for $n=2^k$, where $k\in \mathbb{N}$. We show that \eqref{eq11} holds for $m=2n$.
	\begin{align*}
		f\bigg(\dfrac{\sum_{i=1}^mx_i}{m}\bigg)&=f\bigg(\dfrac{\sum_{i=1}^{2n}x_i}{2n}\bigg)=f\bigg(\dfrac{\frac{1}{n}\sum_{i=1}^nx_i+\frac{1}{n}\sum_{i=n+1}^{2n}x_i}{2}\bigg)\\
		& \leq h(\frac{1}{2})\left(f\left(\dfrac{1}{n}\sum_{i=1}^{n}x_i\right)+f\left(\frac{1}{n}\sum_{i=n+1}^{2n}x_i\right)\right)\\
		& \leq h(\dfrac{1}{2})h(\dfrac{1}{n})\sum_{i=1}^{2n}f(x_i)\\
		& \leq h(\dfrac{1}{2n})\sum_{i=1}^{2n}f(x_i)\\
		& =h(\dfrac{1}{m})\sum_{i=1}^mf(x_i)\,.
	\end{align*}
	Now, we show that if \eqref{eq11} is true for $ n\geq 3 $,  then \eqref{eq11} holds for $ n-1 $. Define
	\begin{align*}
		x_n:=\dfrac{x_1+x_2+\cdots+x_{n-1}}{n-1}\,,
	\end{align*}
	then, we have
	\begin{align*}
		f\bigg(\dfrac{\sum_{i=1}^nx_i}{n}\bigg)&=f\bigg(\dfrac{\sum_{i=1}^{n-1}x_i+\frac{\sum_{i=1}^{n-1}x_i}{n-1}}{n}\bigg)=f\bigg(\dfrac{x_1+x_2+\cdots+x_{n-1}}{n-1}\bigg)\\
		&\leq h(\dfrac{1}{n})\bigg(\sum_{i=1}^{n-1}f(x_i)+f\Big(\dfrac{\sum_{i=1}^{n-1}x_i}{n-1}\Big)\bigg)\,.
	\end{align*}
	So,
	$$
	\bigg(1-h(\dfrac{1}{n})\bigg)f\left(\dfrac{x_1+x_2+\cdots+x_{n-1}}{n-1}\right)\leq h(\dfrac{1}{n})\sum_{i=1}^{n-1}f(x_i)\,.
	$$
	By multiplying two sides of the above inequality to positive number $h(\dfrac{n}{n-1})$, we have
	\begin{align*}
		f\left(\frac{1}{n-1}\sum_{i=1}^{n-1}x_i\right)\le &\bigg(h(\dfrac{n}{n-1})-h(\dfrac{1}{n})h(\dfrac{n}{n-1})\bigg)f\bigg(\dfrac{\sum_{i=1}^{n-1}x_i}{n-1}\bigg)\\
		\leq & h(\dfrac{1}{n})h(\dfrac{n}{n-1})\sum_{i=1}^{n-1}f(x_i)\\
		\leq & h(\dfrac{1}{n-1})\sum_{i=1}^{n-1}f(x_i)\,.
	\end{align*}
	Now, we show that
	\begin{equation}\label{I}
		h(\dfrac{n}{n-1})-h(\dfrac{1}{n})h(\dfrac{n}{n-1})\geq 1\,,
	\end{equation}
	for every $n\in\mathbb{N}$. Since $h$ is super-multiplicative function, we get  
	\begin{equation}\label{II}
		h(\dfrac{1}{n})h(\dfrac{n}{n-1})h(n)\leq h(\dfrac{n}{n-1})\,.
	\end{equation}
	Now, we reduce two sides of inequality from positive value $h(\dfrac{1}{n})h(\dfrac{n}{n-1})$, we have 
	\begin{equation}\label{III}
		h(\dfrac{1}{n})h(\dfrac{n}{n-1})h(n)-h(\dfrac{1}{n})h(\dfrac{n}{n-1})\leq h(\dfrac{n}{n-1})-h(\dfrac{1}{n})h(\dfrac{n}{n-1})\,,
	\end{equation}
	therefore
	\begin{equation}\label{IV}
		h(\dfrac{1}{n})h(\dfrac{n}{n-1})\big(h(n)-1\big)\leq h(\dfrac{n}{n-1})-h(\dfrac{1}{n})h(\dfrac{n}{n-1})\,.
	\end{equation}
	Since $h(t)\geq t$ for $t\in \mathbb{R}^+$, we have
	$$
	h(\dfrac{1}{n})h(\dfrac{n}{n-1})\big(h(n)-1)\big)\geq \dfrac{1}{n}\dfrac{n}{n-1}(n-1)= 1\,.
	$$
	Since \eqref{eq11} holds for every $n=2^k~~(k=1,2,3,\cdots)$, and by using the previous relation, \eqref{eq11} holds for $n=2^k-1$. Now, if assumption of induction holds for $n=2^k-\ell$, then by similar proof, \eqref{eq11} holds for $n=2^k-\ell-1$. Consequently, \eqref{eq11} holds for every $n\in\mathbb{N}$.
	
	\textbf{Case 2.} Let $\lambda_i\in \mathbb{Q}~(i=1,\cdots,n)$ such that $\sum_{i=1}^{n}\lambda_i=1$, then there exist $m\in\mathbb{N}$ and $p_i\in\mathbb{Z}^+$ such that $\lambda_i=\frac{p_i}{m}$, so $\sum_{i=1}^np_i=m$. By Case 1
	\begin{align*}
		f\left(\sum_{i=1}^n\lambda_ix_i\right)=&f\bigg(\dfrac{p_1x_1+p_2x_2+\cdots+p_nx_n}{m}\bigg)\\
		\leq&h(\dfrac{1}{m})\bigg(p_1f(x_1)+p_2f(x_2)+\cdots+p_nf(x_n)\bigg)\\
		=&h(\dfrac{1}{m})\sum_{i=1}^np_if(x_i)\,.
	\end{align*}
	Since $h(p_i)\geq p_i$ and $h$ is super-multiplicative function, we get
	$$
	f\left(\sum_{i=1}^n\lambda_ix_i\right)\leq h(\dfrac{1}{m})\sum_{i=1}^nh(p_i)f(x_i)
	\leq \sum_{i=1}^nh(\dfrac{p_i}{m})f(x_i)=\sum_{i=1}^nh(\lambda_i)f(x_i)\,.
	$$
	Let $\lambda\in [0,1]$, then there exists a sequence of rational numbers as $\{\lambda_n\}_{n=1}^{\infty}\in [0,1]\cap\mathbb{Q}$ such that $\lambda=\lim_{n\to \infty}\lambda_n$. Therefore, for every $n\in\mathbb{N}$ and $x, y\in[a, b]$,
	$$
	f(\lambda_nx+(1-\lambda_n)y)\leq h(\lambda_n)f(x)+h(1-\lambda_n)f(y)
	$$
	So
	\begin{align*}
		f(\lambda x+(1-\lambda)y)&=\lim_{n\to +\infty}f(\lambda_nx+(1-\lambda_n)y)\qquad (f~\text{is continuous})\\
		&\leq \lim_{n\to +\infty}h(\lambda_n)f(x)+\lim_{n\to +\infty}h(1-\lambda_n)f(y)\\
		&=h(\lambda)f(x)+h(1-\lambda)f(y)\qquad (h~\text{is continuous})\,.
	\end{align*}
	This complete the proof.
\end{proof}
\begin{remark}
	If in the previous theorem, the function $h$ is multiplicative, then the theorem holds.
\end{remark}

\begin{theorem}\label{t2.3}
	Let $h:\mathbb{R}^+\to \mathbb{R}^+$ be a super-multiplicative function such that $h(t)\geq t$ and $f$ be a $h$-mid-convex function on $[a,b]$. Then for every $\lambda_i\in\mathbb{Q}\cap [0, 1]$ such that $\sum_{i=1}^n\lambda_i=1$,
	\begin{equation}\label{eI}
		f\left(\sum_{i=1}^n\lambda_ix_i\right)\leq \sum_{i=1}^nh(\lambda_i)f(x_i)\,.
	\end{equation}
\end{theorem}
\begin{proof}
	First, let $\lambda_1=\lambda_2=\frac{1}{2}$. By using $h$-mid-convexity of $f$, we get
	\begin{align*}
		f\left(\frac{x_1+x_2}{2}\right)\leq h(\dfrac{1}{2})(f(x_1)+f(x_2))\,.
	\end{align*}
	
	If $ n=4 $ we have
	\begin{align*}
		f\left(\frac{x_1+x_2+x_3+x_4}{4}\right)=&f\left(\frac{1}{2}\Big(\frac{x_1+x_2}{2}\Big)+\frac{1}{2}\Big(\dfrac{x_3+x_4}{2}\Big)\right)\\
		\leq& h(\frac{1}{2})\left(f\Big(\frac{x_1+x_2}{2}\Big)+f\Big(\dfrac{x_3+x_4}{2}\Big)\right)\\
		\leq& h(\dfrac{1}{4})\sum_{i=1}^4f(x_i)\,.
	\end{align*}
	Similarly, for $n=2^k$, where $k\in\mathbb{N}$,
	\begin{align*}
		f\left(\frac{x_1+x_2+\cdots+x_n}{n}\right) \leq h(\frac{1}{n})\big(f(x_1)+\cdots+f(x_n)\big)\,.
	\end{align*}
	New, for $n=3$, by setting $x_4=\frac{x_1+x_2+x_3}{3}$, we have 
	\begin{align*}
		f\left(\frac{x_1+x_2+x_3}{3}\right)=&f\left(\frac{x_1+x_2+x_3+x_4}{4}\right)\\
		=&f\left(\frac{1}{4}\Big(x_1+x_2+x_3+\dfrac{x_1+x_2+x_3}{3}\Big)\right)\\
		\leq & h(\frac{1}{4})\sum_{i=1}^3f(x_i)+h(\frac{1}{4})f\left(\frac{x_1+x_2+x_3}{3}\right)\,.
	\end{align*}
	So
	$$
	\left(1-h(\dfrac{1}{4})\right)f\left(\dfrac{x_1+x_2+x_3}{3}\right)\leq h(\dfrac{1}{4})\sum_{i=1}^3f(x_i)\,.
	$$
	By multiplication two sides of the above inequality to positive number $h(\frac{4}{3})$, we get
	$$
	h(\frac{4}{3})\Big(1-h(\frac{1}{4})\Big)f\left(\frac{x_1+x_2+x_3}{3}\right)\leq h(\frac{1}{4})h(\frac{4}{3})\sum_{i=1}^3f(x_i)\,.
	$$
	Since $h(t)$ is super-multiplicative function, we have
	$$
	h(\frac{4}{3})\Big(1-h(\frac{1}{4})\Big)f\left(\frac{x_1+x_2+x_3}{3}\right)\leq h(\frac{1}{3})\big(f(x_1)+f(x_2)+f(x_3)\big).
	$$
	Now, we show that
	$$
	h(\frac{4}{3})\Big(1-h(\frac{1}{4})\Big)\geq 1\,.
	$$
	Since $h$ is the super-multiplicative function,
	$$
	h(\frac{1}{4})h(\frac{4}{3})h(4)\leq h(\frac{4}{3})\,.
	$$
	We subtract the positive real number $h(\frac{1}{4})h(\frac{4}{3})$ from both sides of the above inequality, 
	$$
	h(\frac{1}{4})h(\frac{4}{3})(h(4)-1)\leq h(\frac{4}{3})-h(\frac{1}{4})h(\frac{4}{3}).
	$$
	So 
	$$
	h(\frac{1}{4})h(\frac{4}{3})(h(4)-1)\leq h(\frac{4}{3}\Big)\Big(1-h(\frac{1}{4})\Big).
	$$
	Since $h(t)\geq t$, we get
	$$
	h(\frac{4}{3})\Big(1-h(\frac{1}{4})\Big)\geq 1\,.
	$$
	Now, let $n\neq 2^k~(k\in\mathbb{N})$, then there exists $m\in\mathbb{N}$ such that $2^{m-1}<n<2^m$. Hence, we have 
	\begin{align*}
		f\left(\frac{\sum_{i=1}^nx_i}{n}\right)=& 		f\left(\frac{\sum_{i=1}^nx_i}{2^m}+\frac{2^m-n}{2^m}\cdot\frac{\sum_{i=1}^nx_i}{n}\right)\\
		\leq& h(\frac{1}{2^m})\sum_{i=1}^nf(x_i)+h(\frac{2^m-n}{2^m})f\left(\frac{\sum_{i=1}^nx_i}{n}\right)\,.
	\end{align*}
	Therefore, 
	$$
	\left(1-h(\frac{2^m-n}{2^m})\right)f\left(\frac{\sum_{i=1}^nx_i}{n}\right)\leq h(\frac{1}{2^m})\sum_{i=1}^nf(x_i)\,.
	$$
	By multiplication of two sides of the above inequality to positive number $h(\frac{2^m}{n})$ and by using the assumption super-multiplicity of $h$, we have 
	$$
	h(\frac{2^m}{n})\left(1-h(\frac{2^m-n}{2^m})\right)f\left(\frac{\sum_{i=1}^nx_i}{n}\right) \leq h(\frac{1}{n})\sum_{i=1}^nf(x_i)\,.
	$$
	Now, we claim that 
	$$
	h(\frac{2^m}{n})\left(1-h(\frac{2^m-n}{2^m})\right)\geq 1\,.
	$$
	Since $h$ is the super-multiplicative function, we get
	$$
	h\left(\frac{2^m}{2^m-n}\right)h\left(\frac{2^m}{n}\right)h\left(\frac{2^m-n}{2^m}\right)\leq h\left(\frac{2^m}{n}\right)\,.
	$$
	We subtract the positive real number $h\left(\frac{2^m}{n}\right)h\left(\frac{2^m-n}{2^m}\right)$ from both sides of the above inequality, 
	$$
	h\left(\frac{2^m}{n}\right)h\left(\frac{2^m-n}{2^m}\right)\left(h\Big(\frac{2^m-n}{2^m}\Big)-1\right)
	\leq h\left(\frac{2^m}{n}\right)\left(1-h\Big(\frac{2^m-n}{2^m}\Big)\right)\,.
	$$
	Since $ h(t)\geq t $ so
	\begin{align*}
		1=&\frac{2^m}{n}\cdot\frac{2^m-n}{2^m}\left(\frac{2^m}{2^m-n}-1\right)\\
		\leq& h\left(\frac{2^m}{n}\right)h\left(\frac{2^m-n}{2^m}\right)\left(h\Big(\frac{2^m}{2^m-n}\Big)-1\right)\\
		\leq& h\left(\frac{2^m}{n}\right)\left(1-h\Big(\frac{2^m-n}{2^m}\Big)\right)\,,
	\end{align*}
	therefore,
	$$
	f\left(\frac{1}{n}\sum_{i=1}^nx_i\right)\le h\big(\frac{1}{n}\big)\sum_{i=1}^nf(x_i)\,.
	$$
	Now, assume that $\lambda_i\in\mathbb{Q}\cap [a, b]~(i=1,\cdots, n)$ such that $\sum_{i=1}^{n}\lambda_i=1$. So, there exist $u_i\in\mathbb{Z}^+$ and $d\in\mathbb{N}$ such that $\lambda_i=\dfrac{u_i}{d}$. Therefore, $\sum_{i=1}^nu_i=d$. So
	\begin{align*}
		f\left(\sum_{i=1}^n\lambda_ix_i\right)=&f\left(\sum_{i=1}^n\frac{u_ix_i}{d}\right)\\
		=& h\big(\frac{1}{d}\big)\sum_{i=1}^nu_if(x_i)\\
		\le& h\big(\frac{1}{d}\big)\sum_{i=1}^nh(u_i)f(x_i)\\
		\le& \sum_{i=1}^nh\big(\frac{u_i}{d}\big)f(x_i)=\sum_{i=1}^nh(\lambda_i)f(x_i)\,.
	\end{align*}
	The desired inequality followed by an appeal to the special case proved first.
\end{proof}
\section{Some famous Theorems}
In this section, we propound some generalizations for Jensen's theorem, Bernstein-Doetsch theorem, Blumberg-Sierpinski theorem, Ostrowski theorem, and Mehdi theorem for $h$-mid-convex function. 
\begin{lemma}[Jensen's Theorem]
	Let $h:\mathbb{R}\to [0,+\infty)$ be a multiplicative continuous function such that $h(t)\geq t$, $h(0)=0$ and $h(1)=1$. If $f$ is a $h$-mid-convex function and $f$ is bounded above to $M$ on $[a, b]$, then $f$ is continuous.
\end{lemma}
\begin{proof}
	Since the function $f$ is $h$-mid-convex function, so for every $x_i\in (a,b)$, $1\leq i\leq n~(n\in\mathbb{N})$, by Theorem \ref{t2.3},
	$$
	f\left(\frac{x_1+x_2+\cdots+x_n}{n}\right)\leq h(\frac{1}{n})\big(f(x_1)+f(x_2)+\cdots+f(x_n)\big)\,.
	$$
	Choose a positive integer $m$ such that $m<n$ and set 
	$$
	x_1=x_2=x_3=\cdots=x_m:=x+m\delta\,,
	$$
	and
	$$
	x_{m+1}=x_{m+2}=\cdots=x_n:=x\,,
	$$
	where $x\in (a,b)$ and the positive real number $\delta$ is chosen such that $x+n\delta\in (a,b)$. Therefore
	\begin{align*}
		f(x+m\delta)=&f\left(\frac{m(x+n\delta)+(n-m)x}{n}\right)\\
		\leq& h(\frac{m}{n})f(x+n\delta)+h(\frac{n-m}{n})f(x)
	\end{align*}
	By multiplication of two sides of the above inequality to positive number $h(\frac{1}{m})$ and by using the assumption super-multiplicity of $h$, we have
	\begin{align*}
		h\left(\frac{1}{m}\right)&\big(f(x+m\delta)-f(x)+f(x)\big)\\
		&\leq h\left(\frac{1}{n}\right)\big(f(x+n\delta)-f(x)+f(x)\big)+h\left(\frac{n-m}{nm}\right)f(x)\,.
	\end{align*}
	Now, we show that
	\begin{align}\label{eq2.7}
		h\left(\frac{1}{m}\right)&\big(f(x+m\delta)-f(x)\big)\\
		&\leq h\left(\frac{1}{n}\right)\big(f(x+n\delta)-f(x)\big)+f(x)\bigg(h\Big(\frac{1}{n}\Big)
		+h\Big(\frac{n-m}{nm}\Big)-h\Big(\dfrac{1}{m}\Big)\bigg)\nonumber\,.
	\end{align}
	Now, by replacing $\delta$ with $-\delta$ be such that $x-n\delta\in(a, b)$, we have
	\begin{align}\label{eq2.8}
		h\left(\frac{1}{m}\right)&\big(f(x)-f(x-m\delta)\big)\\
		&\geq h\left(\frac{1}{n}\right)\big(f(x)-f(x-n\delta)\big)+f(x)\bigg(h\Big(\frac{1}{m}\Big)
		-h\Big(\frac{1}{n}\Big)-h\Big(\frac{n-m}{nm}\Big)\bigg)\nonumber\,.
	\end{align}
	Due to $h$-mid-convexity of $f$, for every $x\in (a,b)$, we have
	$$
	f(x)=f\left(\frac{x+m\delta+x-m\delta}{2}\right)\leq h(\dfrac{1}{2})\big(f(x+m\delta)+f(x-m\delta)\big)\,.
	$$
	Since $h(t)\geq t$ and $h$ is super-multiplicative, 
	$$
	2f(x)\leq h(2)f(x)\leq f(x+m\delta)+f(x-m\delta)\,.
	$$
	So
	$$
	f(x+m\delta)-f(x)\geq f(x)-f(x-m\delta)\,,
	$$
	By multiplication two sides of the above inequality to positive number $h(\frac{1}{m})$, we have
	\begin{equation}\label{eq2.9}
		h(\dfrac{1}{m})\bigg(f(x+m\delta)-f(x)\bigg)\geq h(\dfrac{1}{m})\bigg(f(x)-f(x-m\delta)\bigg)
	\end{equation}
	Applying \eqref{eq2.7}, \eqref{eq2.8} and \eqref{eq2.9}, we have
	\begin{align*}
		&h(\frac{1}{n})\big(f(x+n\delta)-f(x)\big)+f(x)\left(h\Big(\frac{1}{n}\Big)+h\Big(\frac{n-m}{nm}\Big)-h\Big(\frac{1}{n}\Big)\right)\\
		&\geq h(\frac{1}{m})\big(f(x+m\delta)-f(x)\big)\\
		& \geq  h(\frac{1}{m})\big(f(x)-f(x-m\delta)\big)\\
		& \geq  h(\frac{1}{n})\big(f(x)-f(x-n\delta)\big)+f(x)\left(h\Big(\frac{1}{m}\Big)-h\Big(\frac{1}{n}\Big)-h\Big(\frac{n-m}{nm}\Big)\right)\,.
	\end{align*}
	Since $f$ is bounded above to $M$ on $[a,b]$ and by putting $m=1$, we get
	\begin{align*}
		&h(\frac{1}{n})\big(M-f(x)\big)+f(x)\left(h\Big(\frac{1}{n}\Big)+h\Big(\frac{n-1}{n}\Big)-1\right) \geq f(x+\delta)-f(x)\\
		& \geq f(x)-f(x-\delta)\\
		& \geq h(\frac{1}{n})(f(x)-M)+f(x)\bigg(1-h\Big(\frac{1}{n}\Big)-h\Big(\frac{n-1}{n}\Big)\bigg)\,.
	\end{align*}
	Now, if $\delta\to 0$, $n\in\mathbb{N}$ can be chosen so large that $x\pm n\delta\in (a,b)$, therefore, it can be assumed that $n\to \infty$. So, 
	$$
	\lim_{|\delta|\to 0}|f(x+\delta)-f(x)|=0\,.
	$$
	Hence, the function $f$ is continuous.
\end{proof}
\begin{theorem}[Bernstein-Doetsch Theorem]
	Let $h:\mathbb{R}\to[0,+\infty)$ be a super-multiplicative continuous function and, bounded to $M'$ such that $h(t)\geq t$ and $h(0)=0$ and $f$ be a $h$-mid-convex function on $(a, b)$. If $f$ is bounded above to $M$ in a neighborhood of the point $x_0\in (a, b)$, then $f$ is continuous.
\end{theorem}
\begin{proof}
	Without lose of generality, $x_0=0\in (a,b)$ and $f$ is bounded above to $M$ on $(-r, r)\subseteq (a, b)$, $f(0)=0$. Let $t\in(0, 1)$ be a rational number. Then choose a real number $x$ such that $|x|\leq tr$. We have 
	$$
	x=(1-t)0+t(\frac{x}{t})\,,
	$$
	and
	$$
	0=\frac{t}{1+t}(-\frac{x}{t})+\dfrac{x}{1+t}\,.
	$$
	Since $f$ is $h$-mid-convex function and $f(0)=0$, we get
	\begin{equation}\label{eq33.4}
		f(x)=f\left((1-t)0+t\Big(\frac{x}{t}\Big)\right) \leq h(t)f\Big(\frac{x}{t}\Big)\leq h(t)M\,,
	\end{equation}
	and
	\begin{align*}
		0=&f(0)=f\left(\frac{-tx}{t(1+t)}+\frac{x}{1+t}\right)\\
		\leq & h\Big(\frac{t}{1+t}\Big)f\left(\frac{-x}{t}\right)+h\Big(\frac{1}{1+t}\Big)f(x)\\
		\leq & h\Big(\frac{t}{1+t}\Big)M+h\Big(\frac{1}{1+t}\Big)f(x)\,.
	\end{align*}
	So
	\begin{equation}\label{eq33.5}
		-h\Big(\frac{t}{1+t}\Big)M\geq  h\Big(\frac{1}{1+t}\Big)f(x)\,.
	\end{equation}
	Therefore, $f(x)\geq -h(t)M$, and by using \eqref{eq33.4} and \eqref{eq33.5}, we have $|f(x)|\leq h(t)M$. This shows that $f$ is continuous at $x=0$. Since $f$ is bounded from above on $N_r(0)$, we will show that $f$ is bounded above in a neighborhood of any other point $y\neq 0\in(a,b)$.  Choose a rational number $s>1$ such that $z=sy\in(a,b)$ and set $\lambda:=\dfrac{1}{s}$. Then 
	$$
	M=\big\{v\in\mathbb{R}~:~ v=(1-\lambda)x+\lambda z;~ x\in (-r,r)\big\}\,,
	$$
	is a neighborhood of $ y=\lambda z $ with radius $ (1-\lambda)r $. Moreover,
	\begin{align*}
		f(v)&\leq h(1-\lambda)f(x)+h(\lambda)f(z)\\
		&\leq M'f(x)+M'f(z)\\
		&\leq M'M+M'f(z)\,,
	\end{align*}
	that is, $f$ is bounded above on $U$ and by the Jensen's theorem, $f$ is continuous at $y$.
\end{proof}
\begin{theorem}[Blumberg-Sierpinski Theorem]
	If $h:\mathbb{R}\to [0,+\infty)$ is a continuous and super-multiplicative function such that $h(t)\geq t$ and $f:(a, b)\to\mathbb{R}$ is measurable and $h$-mid-convex function, then $f$ is continuous on (a,b).
\end{theorem}
\begin{proof}
	We use proof by contradiction. Let $f$ be a discontinuous in a point $x_0\in (a,b)$ and $c>0$ such that $(x_0-2c,x_0+2c)\subseteq (a,b)$. Set
	$$
	B_n:=\left\{x\in (a,b)~:~ f(x)>\frac{n}{h(\frac{1}{2})}\right\},\qquad (n\in\mathbb{N})\,.
	$$
	It is obvious that $B_n$ is a measurable set. For fixed $n$, with the Bernstien-Doetsch theorem, $B_n\neq\varnothing$. So, take $u\in B_n\cap (x_0-c,x_0+c)$. Now, for any $\lambda\in [0,1]$
	$$
	\frac{n}{h(\frac{1}{2})}<f(u)=f\left(\frac{u+\lambda c}{2}+\frac{u-\lambda c}{2}\right)\leq h(\frac{1}{2})\Big(f(u+\lambda c)+f(u-\lambda c)\Big)\,.
	$$
	It follows that either $\frac{n}{h(\frac{1}{2})}\leq h(\frac{1}{2})\big(2f(u+\lambda c)\big)$ or $\frac{n}{h(\frac{1}{2})}\leq h(\frac{1}{2})\big(2f(u-\lambda c)\big)$. Since $h(t)\geq t$ and it is super-multiplicative, we have $\frac{n}{h(\frac{1}{2})}\leq f(u+\lambda c)$ or $\frac{n}{h(\frac{1}{2})}\leq f(u-\lambda c)$. Thus, either $u+\lambda c$ or $u-\lambda c$ belongs to $B_n $, or equivalently, if $M_n:=\{x~:~ x=y-u,~ y\in B_n\}$, then either $\lambda c$ or $-\lambda c$ belongs to $M_n$ for any $\lambda\in [0,1]$. We claim that $c\leq m(M_n)=m(B_n)$. Because, set $A_1:=M_n\cap [-c,0]$, $A_2:=M_n\cap [0,c]$. Then $-A_1\cap A_2=[0,c]$. So
	$$
	c=m[0,c]\leq m(-A_1)+m(A_2)=m(A_1)+m(A_2)=m(A_1\cup A_2)\leq m(M_n)\,.
	$$
	On the other hands, $B_1\supseteq B_2\supseteq B_3\supseteq \cdots$, therefore
	$$
	c\leq \lim_{n\to +\infty}m(B_n)=m\left(\bigcap_{i=1}^{+\infty}B_i\right)\,,
	$$
	and so $\cap_{i=1}^{+\infty}B_i\neq \emptyset$. This means that there is a point $ t\in (a,b) $ such that $f(t)>\frac{n}{h(\frac{1}{2})}$ for every $n\in\mathbb{N}$, and this impossible.
	
	To complete the proof, we consider a measurable set $ M $ for which $ -\lambda c $ or $ \lambda c $ is in $ M $ for each $ \lambda\in [0,1] $ 
\end{proof}
\begin{theorem}[Ostrowski Theorem]
	If $h:\mathbb{R}\to [0,+\infty)$ is a super-multiplicative function such that $h(t)\geq t$, $h(1)=1$ and $f:(a,b)\to \mathbb{R}$ is $h$-mid-convex function and $M\subseteq (a,b)$ has positive Lebesgue measure in $(a,b)$ such that $f$ is bounded above on $M$, then $f$ is continuous.
\end{theorem}
\begin{proof}
	Without loss of generality, we assume that $f\leq 0$ on $M$. Since $m(M)>0$, we can cover $M$ by disjoint open intervals as $T_n$ such that 
	$$
	m(M)\leq \sum_{n=1}^{+\infty}m(T_n)<\frac{4}{3}m(M)\,.
	$$
	For every $n\in\mathbb{N}$, the sets $M\cap T_n$ are measurable and disjoint, therefore
	$$
	\sum_{n=1}^{+\infty}m(M\cap T_n)= m\left[\bigcup_{n=1}^{+\infty}m(M\cap T_n)\right]=m(M)\,,
	$$
	hence, there exists at last a $T_n$, we denote by $T$, such that 
	\begin{equation}\label{eII}
		0<m(M\cap T)\leq m(T)<\dfrac{4}{3}m(M\cap T)\,.
	\end{equation}
	Assume that $x_0$ is the mid-point of $T=(c,d)$ and $f$ is not bounded from above in any neighborhood of $x_0$. Choose $x_1$ such that 
	$$
	|x_1-x_0|<\dfrac{m(M\cap T)}{6}=\varepsilon,\qquad f(x_1)>1.
	$$
	Set $N:=\{y~:~ y=2x_1-x;~ x\in M\cap T\}$. It is obvious that $N$ is a reflection and $M\cap T$ is transferred. So
	$$
	m(N)=m(M\cap T)=6\varepsilon.
	$$
	Now, for every $y\in N$
	\begin{align*}
		|y-x_0|&=|2x_1-x-x_0|=|2(x_1-x_0)+x_0-x|\\
		&\leq 2|x_1-x_0|+|x_0-x|\\
		&<2\varepsilon +|x_0-x|\\
		&=\dfrac{m(M\cap T)}{3}+|x_0-x|.
	\end{align*}
	Hence, $N\subset (c-2\varepsilon,d+2\varepsilon)$. Since $m(N)=m(M\cap T)=6\varepsilon$, we have $m(N\cap T)\geq 2\varepsilon$. By $h$-mid-convexity of $f$, 
	$$
	f(x_1)=f\left(\frac{x+y}{2}\right)\leq h(\frac{1}{2})\big(f(x)+f(y)\big)\leq h(\frac{1}{2})f(y)\,.
	$$
	By multiplication two sides of the above inequality to positive number $h(2)$ and using super-multiplicity of $h$, we have
	$$
	h(2)f(x_1)\leq h(2)h(\frac{1}{2})f(y)\leq h(1)f(y)=f(y\,).
	$$
	Since $h(t)\geq t$
	\begin{align*}
		2\leq 2f(x_1)\leq h(2)f(x_1)\leq h(1)f(y)=f(y)\,.
	\end{align*}
	So, $f(y)\geq 2$ and $N\cap M=\varnothing$. Therefore
	\begin{align*}
		m(T)&\geq m\bigg[(M\cap T)\cup (N\cap T)\bigg]\\
		&=m(M\cap T)+m(N\cap T)\\
		&\geq 6\varepsilon +2\varepsilon=8\varepsilon\\
		&=\dfrac{8}{6}m(M\cap T)=\dfrac{4}{3}m(M\cap T).
	\end{align*}
	This is a contradiction by \eqref{eII}. So $f$ is bounded above on a neighborhood of $x_0$. So, $f$ is continuous.
\end{proof}
\begin{corollary}[Mehdi Theorem]
	Let $f$ be a $h$-mid-convex function defined on $I$ such that bounded above on a second category Baire set. Then $f$ is continuous (hence $h$-convex) on $I^\circ$, where $I^\circ$ is interior of $I$.
\end{corollary}

\end{document}